%% file: elsarticle-template-III.tex
\documentclass[final,5p,times,twocolumn]{elsarticle}

\usepackage{graphicx}
\usepackage{subcaption}
\usepackage{amssymb, amsmath}
\usepackage{amsthm}
\usepackage{color}
\usepackage{array}
\usepackage{tabu}
\usepackage{enumitem}

\usepackage[utf8]{inputenc}
\usepackage[english]{babel}
\usepackage[T1, T2A]{fontenc}

\biboptions{compress}

\newtheorem{theorem}{Theorem}[section]
\newtheorem{lemma}[theorem]{Lemma}

\theoremstyle{definition}

\theoremstyle{remark}

\begin{document}

\begin{frontmatter}

\title{On topological obstructions to global stabilization\\ of an inverted pendulum}



\author[label5]{Ivan Polekhin}
\address[label5]{Steklov Mathematical Institute of the Russian Academy of Sciences, Moscow, Russia}


\begin{abstract}
We consider a classical problem of control of an inverted pendulum by means of a horizontal motion of its pivot point. We suppose that the control law can be non-autonomous and non-periodic w.r.t. the position of the pendulum. It is shown that global stabilization of the vertical upward position of the pendulum cannot be obtained for any Lipschitz control law, provided some natural assumptions. Moreover, we show that there always exists a solution separated from the vertical position and along which the pendulum never becomes horizontal. Hence, we also prove that global stabilization cannot be obtained in the system where the pendulum can impact the horizontal plane (for any mechanical model of impact). Similar results are presented for several analogous systems: a pendulum on a cart, a spherical pendulum, and a pendulum with an additional torque control.
\end{abstract}


\end{frontmatter}

\section{Introduction}

One and the same property of a system, considered in different contexts, can both be useful, and appear as an undesirable limitation: possible stability of an inverted pendulum to arbitrary horizontal movements of its pivot point \cite{polekhin2014examples,polekhin2015forced} turns out to be related to the impossibility of global stabilization of a given position or motion of the pendulum.

The problem of stabilization of the vertical upward position of an inverted pendulum (or of an inverted pendulum on a cart) by means of a horizontal motion on its pivot point (or by a horizontal force, correspondingly) is a well-known problem and has been considered by many authors (see, e.g., \cite{mori1976control, henders1992large, lin1992hybrid, chung1995nonlinear, ludvigsen1999stabilization, shiriaev1999global, shiriaev1999swinging, aastrom2000swinging, angeli2001almost, zhao2001hybrid, spong2001nonlinear, shiriaev2004swinging, chatterjee2002swing, gordillo2008new}). This is, among other things, due to the possible applications in real-life systems \cite{sugihara2002real,furuta2003control,kajita2003resolved,popovic2004angular,kajita2006biped}.

It was proved \cite{bhat2000topological} that if the configuration space of a control system has non-trivial topology, then the system cannot have a globally asymptotically stable equilibrium. To be more precise, if the configuration space is closed (compact without boundary), then global stabilization cannot be obtained. One can compare this result with the situation when relatively complex topology of the configuration space leads to non-integrability of a Hamiltonian system \cite{kozlov1983integrability}. For instance, since the configuration space of the spherical pendulum is $\mathbb{S}^2$, the problem of global stabilization of the controlled spherical pendulum can be solved only by means of a non-continuous control \cite{ludvigsen1999stabilization}.

For the system `pendulum on a cart' (its phase space is $\mathbb{S} \times \mathbb{R}^3$), it is also impossible to find such a continuous control that the system would have a globally asymptotically stable equilibrium position \cite{shiriaev1999global,auckly2001mathematical,bhat2000topological}. Even the problem of stabilization of the vertical position of a one-degree-of-freedom simple inverted pendulum does not allow continuous autonomous control which would asymptotically lead the pendulum to the vertical from any initial position. This follows from the fact that a continuous function on a circle, which takes values of opposite sign, has at least two zeros, i.e., the system has at least two equilibria (see system (\ref{main}) below).

The following questions naturally arise. First, do the above statements remain true if we consider the pendulum only in the positions where its mass point is above the pivot point (often there exists a physical constraint in the system which does not allow the rod to be below the plane of support and it is meaningless to consider the pendulum in such positions). Second, is it true that global stabilization cannot be obtained when the control law is a time-dependent function and it is also a non-periodic function of the position of the pendulum? For a relatively broad class of problems, which may appear in practice, we show that for the both questions the answers are positive.

The main results of the paper can be described in the following way. For all systems considered in the paper it was shown \cite{bhat2000topological} that they do not possess a globally asymptotically stable equilibrium and this result follows from the fact that a closed manifold cannot be contractible. At the same time, if we restrict our consideration to a contractible subset of the configuration space of the system, then there exists a vector field with a unique asymptotically stable equilibrium. However, due to limitations caused by the realization of the control mechanism, in real systems we cannot arbitrarily choose the right-hand side of the control system. In particular, we show that for the inverted pendulum there exists a contractible subset of the configuration space such that the vertical upward position belongs to this set, yet this equilibrium is never a global attractor. The existence of such a set is a consequence of our method of control --- we try to stabilize the rod by means of a horizontal motion of the pivot point.

To be more precise, we prove that there exists a solution that does not tend to the equilibrium and the rod never becomes horizontal along it. Note that this is a solution of the system without any additional constraints. Such systems have been considered previously by many authors (see, for instance, \cite{mori1976control, chung1995nonlinear, shiriaev1999swinging, aastrom2000swinging, angeli2001almost}). Let us now suppose that the pivot point of our pendulum is moving on a horizontal plane of support, i.e., the rod is constrained not to be below the horizon. The above mentioned solution still remains in the constrained system, regardless of the model of the rod-plane impact interaction. Therefore, we can claim that global stabilization cannot be obtained for the constrained system, possibly non-continuous.

The proofs are illustrative and based on the Wa\.{z}ewski method \cite{wazewski1947principe, reissig1963qualitative} and similar to the ones in \cite{polekhin2014examples,polekhin2015forced,polekhin2016forced}, where the following system has been studied. Let us consider an inverted pendulum in a gravitational field with its pivot point moving along a horizontal line according to a given law of motion. It was proved that, for an arbitrary smooth function, which describes the motion of the pivot point, there always exists a solution such that the pendulum never becomes horizontal along it (never falls). If the law of motion of the pivot point is periodic, then there exists a periodic solution without falling. We add that similar results can be obtained by means of the variational approach \cite{bolotin2015calculus}.

The paper contains two main sections. In one section we consider in detail the case of control of a simple inverted pendulum (system with one degree of freedom), in another section we consider the controlled spherical pendulum and the pendulum on a cart and also present results on the impossibility of global stabilization.

\section{Simple inverted pendulum}

Consider the following control system
\begin{equation}
\label{main}
\begin{aligned}
&\dot q = p,\\
&\dot p = u(q, p, t) \cdot \sin q - \cos q.
\end{aligned}
\end{equation}

Here and below by $u \in \mathrm{Lip}(\mathbb{R}^3, \mathbb{R})$ we denote the feedback control law for the system, defined by some function from $\mathbb{R}^3$ to $\mathbb{R}$ which is continuous and locally Lipschitz (i.e., Lipschitz for any compact $K \subset \mathbb{R}^3$) in all variables except $t$. System (\ref{main}) describes the motion of a pendulum when the acceleration of its pivot point is given by the function $u$. The coordinate is chosen so that $q = 0$ and $q = \pi$ correspond to the horizontal positions of the rod, $q = \pi/2$ corresponds to its vertical upward position. Without loss of generality, we assume that the mass of the pendulum, its length and the gravity acceleration equal $1$. Note that we do not assume that $u$ is periodic in $q$.

Suppose that we are looking for a control that would stabilize system (\ref{main}) in a vicinity of a certain equilibrium position in the following sense. Let $M$ be a subset of the phase space of the system such that the points of $M$ correspond to the positions of the pendulum in which its rod is above the horizontal line (in our case, $M = \{ 0 < q < \pi \}$) and $\mu \in M$ is the equilibrium for a given control $u$. We assume that the control function $u$ is chosen in such a way that there exists a closed subset $U \subset M$, $\mu \in U \setminus \partial U$ and a  $C^1$-function $V \colon U \to \mathbb{R})$ with the following properties
\begin{enumerate}
    \item[L1. ] $V(\mu) = 0$ and $V > 0$ in $U \setminus \mu$.
    \item[L2. ] Derivative $\dot V$ with respect to system (\ref{main}) is negative in $U \setminus \mu$ for all $t$.
\end{enumerate}
Since the function $V$ can be considered as a Lyapunov function for our system, the equilibrium $\mu$ is stable. If the following (stronger) condition holds
\begin{enumerate}
    \item[L3. ] $\dot V(x,t) \leqslant -W(x) < 0$ in $U \setminus \mu$ for all $t$ and $V(0,t) = W(0) = 0$, where $W \in C(U, \mathbb{R})$,
\end{enumerate}
then $\mu$ is asymptotically stable. For instance, such a function exists in the following case. Suppose that for a given $u$, system (\ref{main}) can be written as follows in a vicinity of $\mu$
\begin{equation*}
\dot x = Ax + f(x,t),
\end{equation*}
where $x = (q,p)$, $A$ is a constant matrix and its eigenvalues have negative real parts, $f$ is a continuous function and $f(t,x) = o(\| x \|)$ uniformly in $t$. Then there exists \cite{demidovich1967lectures} a function $V$ satisfying properties L1, L3.

We now show that in this case the control cannot be global. To be more precise, the following proposition holds

\begin{theorem}
\label{th21}
 Let $u(q,p,t) \in \mathrm{Lip}(\mathbb{R}^3, \mathbb{R})$ be a given control function, $\mu \in M$ be an equilibrium for system (\ref{main}) and $t_0 \in \mathbb{R}$. Suppose there exists a Lyapunov function $V$ satisfying L1 and L2, then there exists an initial condition $(q_0,p_0)$ for $t = t_0$ and an open neighborhood $B \subset M$ of $\mu$ such that, on the interval of existence, the solution $(q(t,q_0,p_0),p(t,q_0,p_0))$ remains in $M \setminus B$.
\end{theorem}

\begin{proof}
For any $C^1$ function $f$ from $\mathbb{R}^n$ to $\mathbb{R}$ such that $f > 0$ everywhere except one point (where $f = 0$), any level set $f = \varepsilon$, for small $\varepsilon > 0$, is a homotopy sphere \cite{wilson1967structure}, and hence a sphere $\mathbb{S}^{n-1}$.

\begin{figure}[h!]
  \centering
    \def\svgwidth{240 pt}
    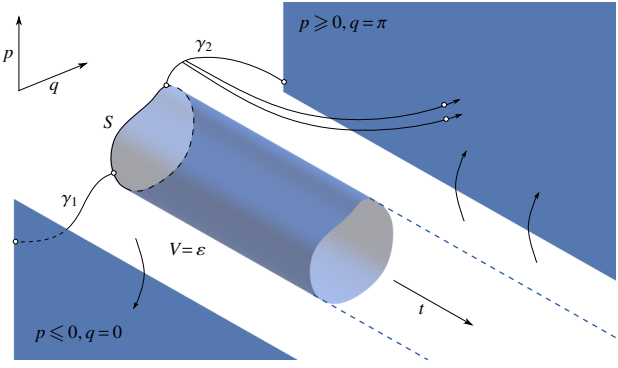
\caption{Exit sets for $(M \setminus B) \times \mathbb{R}^+$. Solutions are externally tangent to $M \times \mathbb{R}^+$ at the points where $p = 0$.}
\label{fig:fig}
\end{figure}

In our case, for small $\varepsilon > 0$, the set $V = \varepsilon$ is a circle (topologically) in the phase space. We shall denote it by $S$ and the corresponding ball by $B$.

Let us consider a curve $\gamma_1$ in the phase space which connects $S$ with the set $\{ q = 0, p < 0 \}$. Similarly, let $\gamma_2$ be a curve connecting
$S$ with the set $\{ q = \pi, p > 0 \}$ and $\gamma_1 \cap \gamma_2 = \varnothing$ (Fig.~1). Any solution starting in $M \setminus B$ at moment $t_0$ can leave the set $(M \setminus B) \times \mathbb{R}^+$ only through one of the following sets of the extended phase space: $S \times \mathbb{R}^+$, $\{ q = 0, p \leqslant 0 \} \times \mathbb{R}^+$ or $\{ q = \pi, p \geqslant 0 \} \times \mathbb{R}^+$. Here by $\mathbb{R}^+$ we denote the set $\{ t \geqslant t_0 \} \subset \mathbb{R}$.

Suppose that all solutions starting in $(S \cup \gamma_1 \cup \gamma_2) \times \{t_0\}$ leave $(M \setminus B) \times \mathbb{R}^+$. If it is true, then for every point $(q,p,t_0) \in (S \cup \gamma_1 \cup \gamma_2) \times \{t_0\}$ there is the point of first exit of the corresponding solution from $(M \setminus B) \times \mathbb{R}^+$. This point belongs to one of the above three sets (Fig.~1). Therefore, we have a map $\sigma$ from the set $(S \cup \gamma_1 \cup \gamma_2) \times \{t_0\}$ to the exit set of $(M \setminus B) \times \mathbb{R}^+$. Note that $\sigma = \mathrm{id}$ on $S \times \{t_0\} \cup (\gamma_1 \cap \{ q = 0, p < 0 \}) \times \{t_0\} \cup (\gamma_2 \cap \{ q = \pi, p > 0 \}) \times \{t_0\} $, i.e. for any point $(q_0, p_0, t_0)$ that belongs to this set, we have $\sigma(q_0, p_0, t_0) = (q_0, p_0, t_0)$. When $(q_0, p_0, t_0) \in S$, it follows from the definition of $S$. For the sets where $q = 0$, $p < 0$ and $q = \pi$, $p > 0$ it immediately follows from the first equation of system (\ref{main}).

Now we prove that $\sigma \colon (S \cup \gamma_1 \cup \gamma_2) \times \{t_0\} \to (S \cup \{ q = 0, p \leqslant 0 \} \cup \{ q = \pi, p \geqslant 0 \} ) \times \mathbb{R}^+$ is a continuous map. First, since the right-hand side of system~(\ref{main}) is Lipschitzian, then the theorem on continuous dependence of solutions on initial data holds. Let us now prove that for any $(q_0, p_0, t_0) \in (S \cup \{ q = 0, p < 0 \} \cup \{ q = \pi, p > 0 \} ) \times \mathbb{R}^+$ there exists $\delta > 0$ such that $(q(t_0 + t, q_0, p_0), p(t_0 + t, q_0, p_0), t_0 + t) \notin (M \setminus B) \times \mathbb{R}^+$. As above, it follows from the definition of $S$ and from equation $\dot q = p$ of the system. Finally, let us show that a solution starting at $\gamma_1$ or $\gamma_2$ cannot leave $(M \setminus B) \times \mathbb{R}^+$ through the points where $p = 0$. Consider the point $(q_0 = 0, p_0 = 0, t_0)$. For the solution starting at this point, we have $q(t_0) = 0$, $\dot q(t_0) = 0$ and $\ddot q(t_0) = -1$ and the corresponding solution is externally tangent to $(M \setminus B) \times \mathbb{R}^+$. Therefore, if our solution passes through the point where $q = 0$ and $p = 0$, then its trajectory is already outside $(M \setminus B) \times \mathbb{R}^+$. The case of the point $(q_0 = 0, p_0 = \pi, t_0)$ can be considered similarly. Therefore, it can be seen that if some solution leave $(M \setminus B) \times \mathbb{R}^+$, then all solutions close to the considered one also leave this set. Moreover, all these solutions leave $(M \setminus B) \times \mathbb{R}^+$ in close points (Fig.~1).

Finally, if our assumption that all solutions starting in $(S \cup \gamma_1 \cup \gamma_2) \times \{t_0\}$ leave $(M \setminus B) \times \mathbb{R}^+$ is true, then we obtain a continuous map between the connected set $S \cup \gamma_1 \cup \gamma_2$ and a disconnected set ($S$, $\gamma_1 \cap \partial M$ и $\gamma_2 \cap \partial M$). In order to construct such a function, we can consider compositions of $\sigma$ with the following maps: the continuous constant maps $\pi_1 \colon \{ q = 0, p \leqslant 0 \} \times \mathbb{R}^+ \to \{\gamma_1 \cap \partial M\} \times \{ t_0 \}$, $\pi_2 \colon \{ q = \pi, p \geqslant 0 \} \times \mathbb{R}^+ \to \{\gamma_2 \cap \partial M\} \times \{ t_0 \}$ and the canonical projection $\pi_3 \colon S \times \mathbb{R}^+ \to S$. The contradiction proves the proposition.

\end{proof}

From the proof it can be seen that we obtain not a single solution that does not leave the set $M \setminus B$, but a one-parameter family of such solutions. This family can be constructed by varying the paths $\gamma_1$ and $\gamma_2$ considered in the proof (Fig.~2).
\begin{figure}[h!]
  \centering
    \def\svgwidth{200 pt}
    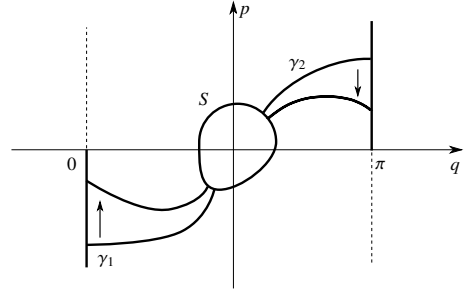
\caption{One-parameter family of initial conditions $S \cup \gamma_1 \cup \gamma_2$ can be obtained by varying $\gamma_1$ and $\gamma_2$. }
\label{fig:fig}
\end{figure}

We would like to note that Theorem~\ref{th21} is proved for the system where $q \in \mathbb{R}$, i.e., we do not impose any constraints on the position of the rod. What we obtain is that there exists a solution that is separated from the equilibrium $\mu$ and along this solution the pendulum never becomes horizontal. In particular, suppose now that our system is a control system with impacts, i.e., we allow the pendulum to fall on the horizontal plane. Since the constraint do not affect the non-stabilized solution, we can conclude that, for any mechanical model of the impact, it is impossible to globally stabilize the rod in a given position.

We add that the solutions can be continued for all $t$ if we assume that  $u(q,p,t)\leqslant a(t)|q| + b(t)|p| + c(t)$, for some continuous functions $a$, $b$ and $c$. For instance, if $u$ is bounded (this assumption is natural, since we always have some power limitations), then the solutions exist for all $t$.

Theorem~\ref{th21} still holds if we consider a more general system

\begin{equation}
\label{main2}
\begin{aligned}
&\dot q = p,\\
&\dot p = u(q, p, t) \cdot \sin q - \cos q + w(q,p,t),
\end{aligned}
\end{equation}

where $w \in \mathrm{Lip}(\mathbb{R}^3, \mathbb{R})$ and $w(0,0,t) < 1$, $w(\pi,0,t) > -1$. Therefore, the system cannot be globally stabilized in the above sense even when there is the control torque $w$ applied at the pivot point. It follows from the fact that, when $q = 0$ and $p = 0$, we have $\ddot q < 0$ and the trajectories of solutions are externally tangent to the set $0 \leqslant q \leqslant \pi$ in the considered points. Similarly, $\ddot q > 0$ for all $t$ when $q = \pi$ and $p = 0$.

Some qualitative properties of system (\ref{main}) can be proved without the assumption on the existence of the function $V$ satisfying L1 and L2. The following result can be proved in essentially the same way as Theorem~\ref{th21}.

\begin{theorem}
\label{th22}
 For any $u(q,p,t) \in \mathrm{Lip}(\mathbb{R}^3, \mathbb{R})$, there exists an initial condition $(q_0,p_0)$ for $t = t_0$ such that on the interval of existence the solution $(q(t,q_0,p_0),p(t,q_0,p_0))$ remains in $\{ 0 < q < \pi \}$.
\end{theorem}
In other words, for any control function, there always exists a solution along which the pendulum never falls.

Moreover, if we assume that $u$ is a bounded $T$-periodic function of $t$ and there is a viscous friction in the system, i.e. the dynamics is described by the following equations

\begin{equation}
\label{fric}
\begin{aligned}
&\dot q = p,\\
&\dot p = u(q, p, t) \cdot \sin q - \cos q - \nu p,
\end{aligned}
\end{equation}
then a result, similar to that for an inverted pendulum without control \cite{polekhin2014examples,polekhin2015forced}, holds.

We will say that $W \subset \mathbb{R}^3$ is a \textit{$T$-periodic segment} for (\ref{fric}) if $W = W_0 \times [0, T]$, where $W_0 \subset \mathbb{R}^2$ is a compact set. The point $(t_0, q_0, p_0) \in W$ is in the exit set $W^-$ if there exists $\delta > 0$ such that $(q(t_0 + t, q_0, p_0), p(t_0 + t, q_0, p_0),t_0 + t) \notin W$ for all $t \in (0, \delta)$. If $W^- = W^{-}_0 \times [0,T] \cup (W \cap \{ t = T \})$, where $W^{-}_0 \subset \mathbb{R}^2$, then $W$ is a simple $T$-periodic segment and by $W^{--} = W^{-}_0 \times [0,T]$ we will denote the essential exit set for $W$. In our case, the result of R.~Srzednicki~\cite{srzednicki1994periodic,srzednicki2005fixed} can be formulated in a simplified form as follows:

\begin{lemma}
\label{LemmaSr}
If there exists a simple $T$-periodic segment $W$ for (\ref{fric}), $W^{--}$ is compact and $\chi(W) - \chi(W^{--}) \ne 0$, then there exists a $T$-periodic solution $(q(t, q_0, p_0),p(t, q_0, p_0))$ such that $(q(t, q_0, p_0),p(t, q_0, p_0)) \in W_0 \setminus \partial W_0$ for all $t$.
\end{lemma}

Here, as usual, by $\chi(W)$ and $\chi(W^{--})$ we denote the Euler-Poincar\'{e} characteristics.  Applying this lemma to the above system, we obtain

\begin{theorem}
\label{th23}
 For any bounded and $T$-periodic in $t$ function $u(q,p,t) \in \mathrm{Lip}(\mathbb{R}^3, \mathbb{R})$ and any $\nu > 0$, there exists an initial condition $(q_0,p_0)$ for $t = 0$ such that the solution $(q(t,q_0,p_0),p(t,q_0,p_0))$ of system (\ref{fric}) is $T$-periodic and remains in $\{ 0 < q < \pi \}$ for all $t$.
\end{theorem}
\begin{proof}

Let us show that for the simple $T$-periodic segment
\begin{equation*}
W = \{ q,p,t \colon q \in [0, \pi], p \in [-\rho, \rho], t \in [0,T] \},
\end{equation*}
where $\rho > 0$ is large, its essential exit set is of the following form (Fig.~3)
\begin{equation*}
\begin{aligned}
W^{--} = &\{ q,p,t \colon q = 0, p \in [-p,0], t \in [0,T] \} \cup \\ &\{ q,p,t \colon q = \pi, p \in [0,p], t \in [0,T] \}.
\end{aligned}
\end{equation*}
\begin{figure}[h!]
  \centering
    \def\svgwidth{160 pt}
    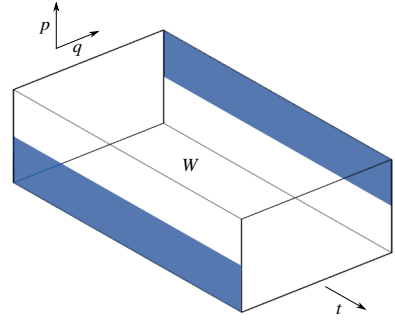
\caption{Sets $W$ and $W^{--}$ (highlighted). }
\label{fig:fig}
\end{figure}
Indeed, if $\rho > 0$ is large, then $\dot p < 0$ for all points where $p = \rho$. Therefore, solutions cannot leave $W$ through the top face of $W$. Similarly, we have $\dot p > 0$ for all points where $p = -\rho$. Finally, we see that $W^{--}$ have the above form and compact. Since $\chi(W) - \chi(W^{--}) \ne -1$, then from Lemma~\ref{LemmaSr} we obtain the existence of a periodic solution.

\end{proof}

From Theorems~\ref{th22} and \ref{th23}, we can also conclude that the problem of global stabilization of the pendulum in a position `below the horizon' cannot be solved by means of a Lipschitz control function $u$.

\section{Systems with two degrees of freedom}
The arguments of the previous section can be carried over to various similar systems. For instance, let us consider the following equations of motion of a controlled inverted spherical pendulum.

\begin{equation}
\label{sphere}
\begin{aligned}
&\ddot \varphi + u \sin\varphi\cos\theta + v\sin\varphi\sin\theta + \dot\theta^2\cos\varphi\sin\varphi = -\cos\varphi,\\
&\ddot \theta \cos^2\varphi - \dot\theta\dot\varphi\sin\varphi\cos\varphi + u \sin\theta\cos\varphi - v\cos\varphi\cos\theta = 0.
\end{aligned}
\end{equation}

Here $\varphi \in (-\pi/2, \pi/2)$ is the inclination angle of the rod, $\theta$ is the azimuth angle. Functions $u, v \in \mathrm{Lip}(\mathbb{R}^5,\mathbb{R})$, $u = u(\varphi, \dot \varphi, \theta, \dot \theta, t)$, $v = v(\varphi, \dot \varphi, \theta, \dot \theta, t)$ are the control accelerations of the pivot point (projections of the acceleration on fixed axes in the horizontal plane). We use the same assumptions concerning the mass and the length of the pendulum and the gravity acceleration as in the previous section.

The configuration space of the system is a two-dimensional sphere. Let $M$ be a subset of the phase space such that the points of $M$ correspond to the positions of the pendulum in which its rod is above the horizontal plane ($\varphi = 0$). Let us suppose that the control functions $u,v$ are chosen in such a way that $\mu \in M$ is an equilibrium of system (\ref{sphere}) and, in a vicinity of $\mu$, there exists a Lyapunov function satisfying L1 and L2. Then global stabilization cannot be achieved for the system. To be more precise, the following holds.

\begin{theorem}
\label{th-shpere}
 Let $u, v \in \mathrm{Lip}(\mathbb{R}^5, \mathbb{R})$ be given control functions, $\mu \in M$ be an equilibrium for system (\ref{sphere}) and $t_0 \in \mathbb{R}$. Suppose there exists a Lyapunov function $V$ satisfying L1 and L2, then there exists an initial condition $(\varphi_0, \dot \varphi_0, \theta_0, \dot\theta_0)$ for $t = t_0$ and an open neighborhood $B \subset M$ of $\mu$ such that on the interval of existence the solution $(q(t,q_0,p_0),p(t,q_0,p_0))$ remains in $M \setminus B$.
\end{theorem}
\begin{proof}
The main idea of the proof is similar to the one in Theorem~\ref{th21}. The only difference is that it is sufficient to connect the sphere $S$ (level set $V = \varepsilon$ for small $\varepsilon > 0$) by one curve $\gamma$ with the set $\{\varphi = 0, \dot \varphi \leqslant 0 \}$.

Now, if we assume that all solutions starting in $\gamma \times \{ t_0 \}$ leave $M \setminus B$, then we obtain a continuous map between a connected set and its two-point boundary ($\gamma \cap S$ and $\gamma \cap \{\varphi = 0, \dot \varphi \leqslant 0 \}$).

The continuity of the corresponding map follows, as in Theorem~\ref{th21}, from the fact that if we put the rod in the horizontal position and $\dot\varphi = 0$, then for any control functions, the pendulum will move to the region where $\varphi < 0$. In other words, if some solution leaves $M \setminus B$, then all close solution also leave this set.\end{proof}

From the proof it can also be seen that we obtain not a single solution, that does not leave the set $M \setminus B$, but a three-parameter family of such solutions: two points in a four-dimensional space can be connected by a three-dimensional family of paths.

The most important generalization of the system considered in the previous section is the controlled system of a pendulum on a cart, which is more correct from the physical point of view than its limiting case, the simple controlled inverted pendulum.

The equations of motion of a pendulum on a cart have the following form
\begin{equation}
\label{cart}
\begin{aligned}
&\dot q = p,\\
&\dot p  = \frac{ u(q,p,x,y,t)\sin q  + p^2\sin q \cos q  - (1+m)\cos q}{m + \cos^2 q},\\
&\dot x = y,\\
&\dot y = (m + \cos^2 q)^{-1}\left( u(q,p,x,y,t) + p^2\cos q - \sin q \cos q \right).
\end{aligned}
\end{equation}

Here $m > 0$ is the mass of the cart, $x$ is the coordinate of the pivot point on the horizontal line, $u \in \mathrm{Lip}(\mathbb{R}^5,\mathbb{R})$ is the horizontal force applied to the cart. We assume that the mass of the pendulum, its length and the gravity acceleration equal $1$.

Note that if the control $u$ does not depend on the position and velocity of the pivot point ($x$ and $y$, correspondingly), then the first two equations can be considered separately. Nonetheless, we will consider the general case, when the control function $u$ is non-autonomous and may depend on the total angular distance covered by the rod $q$ (again, we do not assume that $u$ is periodic in $q$), on the angular velocity of the rod $p$ and on the variables $x$ and $y$, defined above.

\begin{theorem}
\label{th-cart}
 Let $u\in \mathrm{Lip}(\mathbb{R}^5, \mathbb{R})$ be a given control function, $M = \{ 0 < q < \pi \}$,  $\mu \in M$ be an equilibrium for system (\ref{cart})  and $t_0 \in \mathbb{R}$. Suppose there exists a Lyapunov function $V$ satisfying L1 and L2, then there exists an initial condition  $(q_0,p_0,x_0,y_0)$ for $t = t_0$ and an open neighborhood $B \subset M$ of $\mu$ such that the solution starting at $(q_0,p_0,x_0,y_0)$ remains in $M \setminus B$ on the interval of existence.
\end{theorem}
\begin{proof}
From the existence of function $V$ satisfying L1 and L2, we obtain that there exists a ball $B$ such that $\mu \in B$ and $S = \partial B$ is a topological sphere $\mathbb{S}^{3}$. Moreover, any solution starting at $S$ at moment $t_0$ locally leaves $(M \setminus B) \times \mathbb{R}^+$.

Similarly to Theorem~\ref{th21}, we can connect $S$ with the three-dimensional subsets of the phase space $\{ q = 0, p < 0 \}$ and $\{ q = \pi, p > 0 \}$ by non-intersecting curves $\gamma_1$ and $\gamma_2$. Now suppose that all the solutions starting at $(S \cup \gamma_1 \cup \gamma_2) \times \{t_0\}$ leave $(M \setminus B) \times \mathbb{R}^+$.

In order to apply the Wa\.{z}ewski method and show that it is not the case, we have to prove that the corresponding map $\sigma$, that maps $(S \cup \gamma_1 \cup \gamma_2) \times \{t_0\}$ to the boundary of $(M \setminus B) \times \mathbb{R}^+$, is continuous. The key observation is that, like in the case of simple inverted pendulum, $\ddot q > 0$ for all points where $q = \pi$, $p = 0$ and $\ddot q < 0$ for $q = 0$, $p = 0$. Therefore, $\sigma$ is continuous and can be used to construct a continuous map between a connected and a disconnected set. This contradiction completes the proof. \end{proof}

\section{Conclusion and discussion}

 In this section we would like to discuss briefly the most restrictive part of the presented results, the question of existence of the Lyapunov function. As it can be easily seen, everywhere above we have never fully used the fact that $V$ is a Lyapunov function satisfying L1 and L2 (even though, in real-life applications, it is quite natural to assume the existence of such a function). Indeed, what we use is that there exists a `capturing neighborhood' of some point.

If we omit the requirement concerning the existence of a Lyapunov function and do not require the stability in all variables, then, similarly to the theorems above, we can prove the results which state the impossibility of global stabilization in a part of variables. For example, if for the system (\ref{cart}) there exists a closed cylinder $Z = B \times \mathbb{R}^2 \subset M$ such that $\partial B$ is a Jordan curve in the plane $(q,p)$ and any solution of system (\ref{cart}) starting in $\partial Z$ at $t \geqslant t_0$ locally belongs to $Z \setminus \partial Z$. Then there exists a solution which remains in $M \setminus Z$ on the interval of existence. In other words, we do not require the stability in the variables $x$ and $y$. We only assume that any solution starting in $Z$, will stay in this cylinder (w.r.t. to the variables $q$ and $p$). In particular, we do not assume that there exists a stable equilibrium in the system.

For the sake of simplicity of the exposition, in conditions L1 and L2, we consider a Lyapunov function that does not depend on time. However, taking into account the results on the existence of a Lyapunov function in a vicinity of an asymptotically stable equilibrium, it may be useful to consider more complex Lyapunov functions . In particular, let us consider system (\ref{cart}) and suppose that its right-hand side is $T$-periodic in $t$. Let $\mu$ be an asymptotically stable equilibrium. Then there exists a smooth $T$-periodic Lyapunov function $V(q,p,x,y,t)$ satisfying the conditions of the Lyapunov theorem on
asymptotic stability \cite{krasovskii1963stability} and we can prove that global stabilization of the asymptotically stable equilibrium cannot be obtained. The proof is based on the consideration of the $T$-periodic level set $V = \varepsilon$ of the Lyapunov function, for small $\varepsilon > 0$. The function $V$ is defined in a vicinity of the point $\mu$. Any solution starting in the set $V = \varepsilon$ belongs to the set $V < \varepsilon$ for all subsequent $t$. The proof can be obtained by using similar arguments as in Theorem \ref{th21}. Therefore, we have that for systems (\ref{main}), (\ref{sphere}) and (\ref{cart}), if they are periodic in $t$, global stabilization cannot be obtained.

In conclusion, note that the results presented in the paper can be generalized and developed in various ways. For instance, it is possible to apply the same ideas to systems with more than two degrees of freedom. Similar results can be proved for the case when we try to stabilize the system in a vicinity of a given trajectory, which may not be an equilibrium. Also, we can consider systems with dry or viscous friction and systems different from the inverted pendulum, e.g., the control system of a point moving on a surface which intersects the horizontal plane orthogonally. In all these, and many other cases, if the solutions depend continuously on initial data, the same topological obstructions to global stabilization appear and the above methods can be applied. However, these generalizations are out of the scope of this paper and will be developed elsewhere.

\section*{Acknowledgements}
This work was supported by the Russian Science Foundation under grant No. 14-50-00005.

\bibliographystyle{elsarticle-num}

\bibliography{sample}

\end{document}

%% file: control-V-new.pdf_tex
\begingroup%
  \makeatletter%
  \providecommand\color[2][]{%
    \errmessage{(Inkscape) Color is used for the text in Inkscape, but the package 'color.sty' is not loaded}%
    \renewcommand\color[2][]{}%
  }%
  \providecommand\transparent[1]{%
    \errmessage{(Inkscape) Transparency is used (non-zero) for the text in Inkscape, but the package 'transparent.sty' is not loaded}%
    \renewcommand\transparent[1]{}%
  }%
  \providecommand\rotatebox[2]{#2}%
  \ifx\svgwidth\undefined%
    \setlength{\unitlength}{768.01425781bp}%
    \ifx\svgscale\undefined%
      \relax%
    \else%
      \setlength{\unitlength}{\unitlength * \real{\svgscale}}%
    \fi%
  \else%
    \setlength{\unitlength}{\svgwidth}%
  \fi%
  \global\let\svgwidth\undefined%
  \global\let\svgscale\undefined%
  \makeatother%
  \begin{picture}(1,0.56309813)%
    \put(0,0){\includegraphics[width=\unitlength]{control-V-new.pdf}}%
    \put(0.50540043,0.52237592){\color[rgb]{0,0,0}\makebox(0,0)[lb]{\smash{${\scriptstyle p \,\geqslant\, 0,\, q \,= \,\pi}$}}}%
    \put(0.09324924,0.03448276){\color[rgb]{0,0,0}\makebox(0,0)[lb]{\smash{${\scriptstyle p \,\leqslant \,0,\, q\, =\, 0}$}}}%
    \put(0.04343232,0.47496248){\color[rgb]{0,0,0}\makebox(0,0)[lb]{\smash{${\scriptstyle p}$}}}%
    \put(0.11623389,0.42485415){\color[rgb]{0,0,0}\makebox(0,0)[lb]{\smash{${\scriptstyle q}$}}}%
    \put(0.69274621,0.07343289){\color[rgb]{0,0,0}\makebox(0,0)[lb]{\smash{${\scriptstyle t}$}}}%
    \put(0.30232148,0.16563013){\color[rgb]{0,0,0}\makebox(0,0)[lb]{\smash{${\scriptstyle V =\, \varepsilon}$}}}%
    \put(0.20047356,0.36592964){\color[rgb]{0,0,0}\makebox(0,0)[lb]{\smash{${\scriptstyle S}$}}}%
    \put(0.34481972,0.49081536){\color[rgb]{0,0,0}\makebox(0,0)[lb]{\smash{${\scriptstyle \gamma_2}$}}}%
    \put(0.13498227,0.2475196){\color[rgb]{0,0,0}\makebox(0,0)[lb]{\smash{${\scriptstyle \gamma_1}$}}}%
  \end{picture}%
\endgroup%

%% file: control-I.pdf_tex
\begingroup%
  \makeatletter%
  \providecommand\color[2][]{%
    \errmessage{(Inkscape) Color is used for the text in Inkscape, but the package 'color.sty' is not loaded}%
    \renewcommand\color[2][]{}%
  }%
  \providecommand\transparent[1]{%
    \errmessage{(Inkscape) Transparency is used (non-zero) for the text in Inkscape, but the package 'transparent.sty' is not loaded}%
    \renewcommand\transparent[1]{}%
  }%
  \providecommand\rotatebox[2]{#2}%
  \ifx\svgwidth\undefined%
    \setlength{\unitlength}{177.87231445bp}%
    \ifx\svgscale\undefined%
      \relax%
    \else%
      \setlength{\unitlength}{\unitlength * \real{\svgscale}}%
    \fi%
  \else%
    \setlength{\unitlength}{\svgwidth}%
  \fi%
  \global\let\svgwidth\undefined%
  \global\let\svgscale\undefined%
  \makeatother%
  \begin{picture}(1,0.58679733)%
    \put(0,0){\includegraphics[width=\unitlength]{control-I.pdf}}%
    \put(0.57387014,0.44025777){\color[rgb]{0,0,0}\makebox(0,0)[lb]{\smash{${\scriptstyle \gamma_2}$}}}%
    \put(0.21609644,0.06900084){\color[rgb]{0,0,0}\makebox(0,0)[lb]{\smash{${\scriptstyle \gamma_1}$}}}%
    \put(0.40496991,0.36727361){\color[rgb]{0,0,0}\makebox(0,0)[lb]{\smash{${\scriptstyle S}$}}}%
    \put(0.48013945,0.53856746){\color[rgb]{0,0,0}\makebox(0,0)[lb]{\smash{${\scriptstyle p}$}}}%
    \put(0.87597318,0.2456286){\color[rgb]{0,0,0}\makebox(0,0)[lb]{\smash{${\scriptstyle q}$}}}%
    \put(0.15821142,0.24592809){\color[rgb]{0,0,0}\makebox(0,0)[lb]{\smash{${\scriptstyle 0}$}}}%
    \put(0.73580188,0.24723477){\color[rgb]{0,0,0}\makebox(0,0)[lb]{\smash{${\scriptstyle \pi}$}}}%
  \end{picture}%
\endgroup%

%% file: control-VI.pdf_tex
\begingroup%
  \makeatletter%
  \providecommand\color[2][]{%
    \errmessage{(Inkscape) Color is used for the text in Inkscape, but the package 'color.sty' is not loaded}%
    \renewcommand\color[2][]{}%
  }%
  \providecommand\transparent[1]{%
    \errmessage{(Inkscape) Transparency is used (non-zero) for the text in Inkscape, but the package 'transparent.sty' is not loaded}%
    \renewcommand\transparent[1]{}%
  }%
  \providecommand\rotatebox[2]{#2}%
  \ifx\svgwidth\undefined%
    \setlength{\unitlength}{506.4bp}%
    \ifx\svgscale\undefined%
      \relax%
    \else%
      \setlength{\unitlength}{\unitlength * \real{\svgscale}}%
    \fi%
  \else%
    \setlength{\unitlength}{\svgwidth}%
  \fi%
  \global\let\svgwidth\undefined%
  \global\let\svgscale\undefined%
  \makeatother%
  \begin{picture}(1,0.7409234)%
    \put(0,0){\includegraphics[width=\unitlength]{control-VI.pdf}}%
    \put(0.11354971,0.67454753){\color[rgb]{0,0,0}\makebox(0,0)[lb]{\smash{${\scriptstyle p}$}}}%
    \put(0.19204537,0.62248975){\color[rgb]{0,0,0}\makebox(0,0)[lb]{\smash{${\scriptstyle q}$}}}%
    \put(0.80822348,0.00637581){\color[rgb]{0,0,0}\makebox(0,0)[lb]{\smash{${\scriptstyle t}$}}}%
    \put(0.44649785,0.34359568){\color[rgb]{0,0,0}\makebox(0,0)[lb]{\smash{${\scriptstyle W}$}}}%
  \end{picture}%
\endgroup%